\documentclass[11pt,a4paper]{article}

\usepackage[T1]{fontenc}
\usepackage[utf8]{inputenc}
\usepackage{lmodern}
\usepackage{amsmath,amssymb,amsthm,mathtools}
\usepackage[margin=1in]{geometry}
\usepackage{microtype}

\allowdisplaybreaks

\newtheorem{theorem}{Theorem}[section]
\newtheorem{lemma}[theorem]{Lemma}
\newtheorem{proposition}[theorem]{Proposition}

\begin{document}

\begin{center}
{\Large\bfseries Ball Rigidity of Local Minimizing Domains for the Best Fractional Sobolev Constant: The Subquadratic Case\par}
\vspace{1.2em}
{\large Zikang Deng\par}
\vspace{0.35em}
Beijing Normal University
\end{center}

\begin{abstract}
After Corollary 1.3 in \emph{Calculus of Variations and Partial Differential Equations} 60 (2021), Paper 231, Djitte, Fall, and Weth asked whether, when $1<p<2$, a volume-constrained local minimizing domain for the best fractional Sobolev constant must still be a ball. This paper solves that problem. Let $0<s<1$, $1<p<2$, and let $\Omega\subset\mathbb R^N$ be a bounded $C^3$ domain. If $\Omega$ is a local minimizing domain for
\[
\lambda_{s,p}(\Omega)=\inf\bigl\{[u]_s^2:u\in\mathcal H_0^s(\Omega),\ \|u\|_{L^p(\Omega)}=1\bigr\}
\]
under smooth volume-preserving deformations, then $\Omega$ is a ball. The proof first uses the fractional Hadamard formula to reduce shape minimality to the overdetermined boundary condition $u/\delta^s=C_0$. To overcome the moving-plane obstruction caused by the failure of $u^{p-1}$ to be Lipschitz at zero, we establish a weighted singular narrow-domain maximum principle whose absorption factor is exactly the $sp/N$ power of the measure of the negative set. The difficulty at a corner is resolved by a finite boundary expansion: setting $\rho=sp$, the boundary quotient is composed of finitely many constant-coefficient normal powers $\delta^{k\rho}$ and a $C^{1,\varepsilon}$ remainder; when $k\rho=1$, the unique resonant correction is $\delta\log\delta$. This expansion makes all lower-order normal terms on the two sides of an orthogonal corner cancel, thereby yielding the first-order tangential vanishing required by the moving-plane corner lemma. The method does not require the domain to be convex and covers the full range $0<s<1$ and $1<p<2$.
\end{abstract}

\noindent\textbf{Keywords:} fractional Laplacian; shape optimization; Hadamard formula; overdetermined problem; moving planes; boundary expansion

\section{Introduction}

Let $0<s<1$, and take the normalization constant
\begin{equation}
c_{N,s}=\pi^{-N/2}s4^s\frac{\Gamma\!\left(\frac N2+s\right)}{\Gamma(1-s)}.
\tag{1.1}
\end{equation}
For $u\in H^s(\mathbb R^N)$, write
\begin{equation}
[u]_s^2=\frac{c_{N,s}}{2}\iint_{\mathbb R^{2N}}
\frac{(u(x)-u(y))^2}{|x-y|^{N+2s}}\,dx\,dy.
\tag{1.2}
\end{equation}
If $\Omega\subset\mathbb R^N$ is a bounded open set, define
\[
\mathcal H_0^s(\Omega)=\overline{C_c^\infty(\Omega)}^{[\,\cdot\,]_s}
\]
and the best fractional Sobolev constant
\begin{equation}
\lambda_{s,p}(\Omega)=\inf\left\{[u]_s^2:u\in\mathcal H_0^s(\Omega),\ \int_\Omega |u|^p\,dx=1\right\}.
\tag{1.3}
\end{equation}

Djitte, Fall, and Weth~[1] established the fractional Hadamard formula for $\lambda_{s,p}$ and used it to study shape extrema under a fixed-volume constraint. Set
\[
2_s^*=\begin{cases}
\dfrac{2N}{N-2s},&N>2s,\\[0.4em]
\infty,&N\le 2s.
\end{cases}
\]
Within the range $1\le p<2_s^*$ considered in that paper, they proved that, when
\[
p\in\{1\}\cup[2,\infty),
\]
every bounded $C^3$ volume-constrained local minimizing domain is a ball; immediately after that result, they pointed out that the case $1<p<2$ remained an open problem. The difficulty does not come from the shape derivative: in this range the positive minimizer is still unique, and the Hadamard formula still yields an overdetermined condition. The real obstruction is that the nonlinearity $t\mapsto t^{p-1}$ in the Euler--Lagrange equation is not locally Lipschitz at $t=0$, so the original fractional moving-plane theorem~[3] cannot be applied directly with a bounded secant coefficient.

Djitte and Minlend~[2] later proved the corresponding result for convex $C^{1,1}$ domains by means of continuous Steiner symmetrization. This paper removes the convexity assumption and gives the general $C^3$ result required by the original problem.

\medskip
\noindent\textbf{Theorem 1.1 (Main theorem).}
Let $N\ge1$, $0<s<1$, $1<p<2$, and let $\Omega\subset\mathbb R^N$ be a bounded $C^3$ domain. If $\Omega$ is a local minimizing domain for the map
\[
\Omega\longmapsto\lambda_{s,p}(\Omega)
\]
under fixed-volume $C^2$ diffeomorphic deformations, then $\Omega$ is a ball.

\medskip
The main theorem follows from the following rigidity theorem for an overdetermined problem.

\medskip
\noindent\textbf{Theorem 1.2 (Rigidity of the overdetermined problem).}
Let $N\ge1$, $0<s<1$, $0<q<1$, and let $\Omega\subset\mathbb R^N$ be a bounded $C^3$ domain. Suppose that there exist $\Lambda>0$ and a nontrivial function
\[
u\in\mathcal H_0^s(\Omega)\cap L^\infty(\mathbb R^N),\qquad u\ge0,
\]
such that, in the weak sense,
\begin{equation}
\begin{cases}
(-\Delta)^s u=\Lambda u^q,&\text{in }\Omega,\\
u=0,&\text{in }\mathbb R^N\setminus\Omega,\\
\dfrac{u}{\delta^s}=C_0>0,&\text{on }\partial\Omega,
\end{cases}
\tag{1.4}
\end{equation}
where $\delta(x)=\operatorname{dist}(x,\mathbb R^N\setminus\Omega)$ and the boundary condition is understood in the sense of continuous traces. Then $\Omega$ is a ball.

The proof contains two new estimates. The first treats the singular secant coefficient in the moving-plane argument. Set
\[
b=s(1-q).
\]
On the set $A$ where the reflected difference is negative, the secant coefficient can blow up at most like $\delta^{-b}$. We prove that
\begin{equation}
\int_A\delta^{-b}\phi^2\,dx\le C|A|^{(2s-b)/N}\mathcal E_s(\phi,\phi).
\tag{1.5}
\end{equation}
When $q=p-1$, one has $2s-b=sp$, so the singular term can be absorbed whenever the negative set is sufficiently narrow.

The second estimate treats an orthogonal corner. Set
\begin{equation}
\rho=s(1+q)=sp.
\tag{1.6}
\end{equation}
If $\rho>1$, higher-order boundary Schauder estimates directly give $u/\delta^s\in C^{1,\varepsilon}$. If $\rho\le1$, the boundary quotient generally contains normal powers of order below one. We prove that there exist constants $A_k$ and $H\in C^{1,\varepsilon}$ such that
\begin{equation}
\frac{u}{\delta^s}=H+\sum_{\substack{k\ge1\\k\rho<1}}A_k\delta^{k\rho}.
\tag{1.7}
\end{equation}
If there exists a positive integer $M$ such that $M\rho=1$, the unique resonant term $A_*\delta\log\delta$ must also be added to the right-hand side. The coefficients of all lower-order terms are constant along the entire boundary and therefore cancel under tangential reflection. This is more precise than requiring a priori that the entire boundary quotient belong to $C^1$.

The organization of the paper closely follows the order of the argument in~[1]. Section~2 gives the variational problem, uniqueness, and the shape derivative; Section~3 proves the singular narrow-domain maximum principle; Section~4 establishes power formulas at a curved boundary; Section~5 completes the finite boundary expansion; Section~6 proves tangential vanishing; and Section~7 uses moving planes to prove the overdetermined rigidity theorem and the main theorem.

\section{The Variational Problem and the Overdetermined Condition}

\subsection{The energy space and the Euler--Lagrange equation}

Define the bilinear form
\begin{equation}
\mathcal E_s(u,v)=\frac{c_{N,s}}{2}\iint_{\mathbb R^{2N}}
\frac{(u(x)-u(y))(v(x)-v(y))}{|x-y|^{N+2s}}\,dx\,dy.
\tag{2.1}
\end{equation}
Thus $[u]_s^2=\mathcal E_s(u,u)$. When $1<p<2$, compact embedding ensures that the infimum in (1.3) is attained by a nonnegative function. Taking the absolute value does not increase the energy, and the strong maximum principle then implies that the minimizer is strictly positive in $\Omega$. The standard constrained variation gives
\begin{equation}
\mathcal E_s(u,\varphi)=\lambda_{s,p}(\Omega)\int_\Omega u^{p-1}\varphi\,dx,
\qquad \varphi\in\mathcal H_0^s(\Omega).
\tag{2.2}
\end{equation}

\medskip
\noindent\textbf{Lemma 2.1 (Uniqueness of the positive minimizer).}
If $1<p\le2$, then the positive minimizer in (1.3) is unique.

\smallskip
\noindent\textit{Proof.}
Let $u,v$ be two positive $L^p$-normalized minimizers. For $t\in(0,1)$, set
\[
\sigma_t(x)=\bigl((1-t)u(x)^p+tv(x)^p\bigr)^{1/p}.
\]
Clearly $\|\sigma_t\|_p=1$. For arbitrary $x,y$, regard
\[
U_x=\bigl((1-t)^{1/p}u(x),\,t^{1/p}v(x)\bigr)\in\mathbb R^2
\]
as an $\ell^p$ vector, so that $\sigma_t(x)=\|U_x\|_{\ell^p}$. By the reverse triangle inequality,
\begin{align*}
|\sigma_t(x)-\sigma_t(y)|^2
&\le\bigl((1-t)|u(x)-u(y)|^p+t|v(x)-v(y)|^p\bigr)^{2/p}\\
&\le(1-t)|u(x)-u(y)|^2+t|v(x)-v(y)|^2.
\end{align*}
The last step is the power-mean inequality on the probability measure $\{1-t,t\}$; equivalently, it follows by applying Jensen's inequality to the convex function $r\mapsto r^{2/p}$. Integrating with respect to $x,y$ yields
\[
[\sigma_t]_s^2\le(1-t)[u]_s^2+t[v]_s^2=\lambda_{s,p}(\Omega).
\]
Consequently, both pointwise inequalities in the integrand are equalities almost everywhere. The equality conditions show that the vectors $U_x$ and $U_y$ are collinear for almost every pair $(x,y)$, and hence $u/v$ is constant. The $L^p$ normalizations of the two functions then give $u=v$. This also includes the Hilbert-space case $p=2$. \hfill$\square$

\subsection{The Hadamard formula}

Let $\Phi_\varepsilon$ be a $C^2$ family of diffeomorphisms, with $\Phi_0=\operatorname{id}$, and write
\[
X=\left.\partial_\varepsilon\right|_{\varepsilon=0}\Phi_\varepsilon,
\qquad \Omega_\varepsilon=\Phi_\varepsilon(\Omega).
\]
The following shape derivative formula is the specialized form of~[1, Theorem 1.1 and Corollary 1.2]. Here we take $\nu$ to be the inward unit normal; if the outward unit normal is used instead, a minus sign must be placed in front of the integral on the right-hand side.

\medskip
\noindent\textbf{Theorem 2.2 (Fractional Hadamard formula).}
Let $\Omega$ be a $C^{1,1}$ domain, and suppose that $\lambda_{s,p}(\Omega)$ has a unique positive minimizer $u$. Then
\begin{equation}
\left.\frac{d}{d\varepsilon}\right|_{0}\lambda_{s,p}(\Omega_\varepsilon)
=\Gamma(1+s)^2\int_{\partial\Omega}\left(\frac{u}{\delta^s}\right)^2 X\cdot\nu\,d\sigma.
\tag{2.3}
\end{equation}

\medskip
\noindent\textbf{Proposition 2.3 (Shape minimality implies the overdetermined condition).}
Under the assumptions of Theorem 1.1, the unique positive minimizer $u$ satisfies
\begin{equation}
\frac{u}{\delta^s}=C_0>0\qquad\text{on }\partial\Omega.
\tag{2.4}
\end{equation}

\smallskip
\noindent\textit{Proof.}
By Lemma 2.1, the Hadamard formula applies. Fix an arbitrary $h\in C^2(\partial\Omega)$ such that
\[
\int_{\partial\Omega}h\,d\sigma=0.
\]
Extend $h\nu$ to a $C^2$ vector field in a neighborhood of the boundary and make a second-order volume correction. One can then construct a family of diffeomorphisms that preserves volume exactly and whose initial velocity satisfies $X\cdot\nu=h$. Local minimality holds for both positive and negative values of the parameter, so the first derivative vanishes. By (2.3),
\[
\int_{\partial\Omega}\left(\frac{u}{\delta^s}\right)^2h\,d\sigma=0
\]
for every zero-mean $h$. Therefore $(u/\delta^s)^2$ is constant on the boundary. The boundary Hopf principle gives $u/\delta^s>0$, and hence (2.4) follows. \hfill$\square$

\section{A Singular Narrow-Domain Maximum Principle}

Throughout this section, let $q=p-1\in(0,1)$ and set
\begin{equation}
a=sq=s(p-1),\qquad b=s(1-q)=s(2-p),\qquad \rho=s+a=sp.
\tag{3.1}
\end{equation}
By boundary regularity, (2.4), and interior positivity, there exists a constant $m_0>0$ such that
\begin{equation}
u(x)\ge m_0\delta(x)^s,\qquad x\in\Omega.
\tag{3.2}
\end{equation}

\subsection{A weighted support estimate}

\medskip
\noindent\textbf{Lemma 3.1.}
Let $D\subset\Omega$ be measurable, and let $\phi\in\mathcal H_0^s(\Omega)$ satisfy $\phi=0$ on $\Omega\setminus D$. Then
\begin{equation}
\int_D\delta(x)^{-b}\phi(x)^2\,dx
\le C|D|^{(2s-b)/N}\mathcal E_s(\phi,\phi)
=C|D|^{sp/N}\mathcal E_s(\phi,\phi).
\tag{3.3}
\end{equation}
The constant depends only on $N,s,p$ and the Lipschitz character of $\Omega$.

\smallskip
\noindent\textit{Proof.}
Since $0<b<s$, one can choose
\begin{equation}
\frac b2<t<\min\left\{s,\frac12\right\}.
\tag{3.4}
\end{equation}
The fractional Hardy inequality of order $t<1/2$ (with $\phi$ extended by zero on $\Omega^c$) gives
\begin{equation}
\int_\Omega\frac{\phi^2}{\delta^{2t}}\,dx
\le C[\phi]_{\dot H^t(\mathbb R^N)}^2.
\tag{3.5}
\end{equation}
H\"older's inequality and (3.5), in succession, give
\begin{align}
\int_D\delta^{-b}\phi^2
&\le\left(\int_\Omega\delta^{-2t}\phi^2\right)^{b/(2t)}
\left(\int_D\phi^2\right)^{1-b/(2t)}\notag\\
&\le C[\phi]_{\dot H^t}^{b/t}\|\phi\|_2^{2-b/t}.
\tag{3.6}
\end{align}
Interpolation between the homogeneous Sobolev scale and $L^2$ gives
\[
[\phi]_{\dot H^t}^2
\le C\mathcal E_s(\phi,\phi)^{t/s}\|\phi\|_2^{2(1-t/s)}.
\]
Substituting this into (3.6) and collecting the exponents yields
\begin{equation}
\int_D\delta^{-b}\phi^2
\le C\mathcal E_s(\phi,\phi)^{b/(2s)}\|\phi\|_2^{2-b/s}.
\tag{3.7}
\end{equation}
The fractional Faber--Krahn inequality gives, by scaling,
\begin{equation}
\|\phi\|_2^2\le C|D|^{2s/N}\mathcal E_s(\phi,\phi).
\tag{3.8}
\end{equation}
Substitution of (3.8) into (3.7) gives the exponent of the measure
\[
\frac{2s}{N}\left(1-\frac{b}{2s}\right)=\frac{2s-b}{N}.
\]
Finally, (3.1) gives $2s-b=sp$, and the conclusion follows. \hfill$\square$

\subsection{An estimate for the antisymmetric negative part}

Let $H$ be a half-space, and let $Q$ be reflection across $\partial H$. A function $w$ is called antisymmetric if
\[
w(Qx)=-w(x).
\]

\medskip
\noindent\textbf{Lemma 3.2 (Singular narrow-domain principle).}
Let $D\subset H\cap\Omega$ be open, and suppose that its reflection $Q(D)$ is contained in $\Omega$. Let $w\in H^s(\mathbb R^N)$ be antisymmetric with respect to $Q$, suppose that $w\ge0$ on $H\setminus D$, and suppose that $w$ weakly satisfies
\begin{equation}
(-\Delta)^s w=c(x)w\qquad\text{in }D.
\tag{3.9}
\end{equation}
If, on the negative set $A=\{x\in D:w(x)<0\}$,
\begin{equation}
c_+(x)\le C_0\bigl(\delta(x)^{-b}+\delta(Qx)^{-b}\bigr),
\tag{3.10}
\end{equation}
then there exists $\varepsilon_0>0$ such that $w\ge0$ in $D$ whenever $|A|<\varepsilon_0$.

\smallskip
\noindent\textit{Proof.}
Let $\phi=w^-\mathbf 1_H$, extended by zero outside $H$. Since $w\in H^s(\mathbb R^N)$, $w$ is antisymmetric with respect to $Q$, and $w\ge0$ on $H\setminus D$, the antisymmetric negative-part truncation result gives
\[
\phi\in\mathcal H_0^s(D),\qquad \operatorname{supp}\phi\subset\overline A.
\]
This conclusion holds for the full range $0<s<1$. Its proof splits the fractional energy into the four parts $H\times H$, $H\times Q(H)$, and so on, and uses $|r^- -t^-|\le|r-t|$ together with a comparison of the reflected kernels; see~[3, proof of Proposition 3.1]. Therefore $\phi$ may be used as a test function in (3.9).

We now write out the antisymmetric energy calculation. Set
\[
K(x,y)=|x-y|^{-N-2s}.
\]
For $x,y\in H$, one has $|x-y|<|Qx-y|$. Split $\mathbb R^N\times\mathbb R^N$ into the four blocks determined by $H$ and $Q(H)$, and make the reflection change of variables in the integrals containing $Q(H)$. Since $w=w^+-\phi$ in $H$ and $w^+\phi=0$, direct expansion gives
\begin{align}
\mathcal E_s(w,\phi)+\mathcal E_s(\phi,\phi)
&=-c_{N,s}\iint_{H\times H}\phi(y)
\Bigl[w^+(x)\bigl(K(x,y)-K(Qx,y)\bigr)\notag\\
&\hspace{10em}+\phi(x)K(Qx,y)\Bigr],dx,dy\le0.
\tag{3.11}
\end{align}
Consequently,
\begin{equation}
\mathcal E_s(\phi,\phi)\le-\mathcal E_s(w,\phi).
\tag{3.12}
\end{equation}
On the other hand, testing (3.9) with $\phi$ and using $w=-\phi$ on $A$, we obtain
\begin{equation}
-\mathcal E_s(w,\phi)=\int_A c(x)\phi^2\,dx
\le\int_A c_+(x)\phi^2\,dx.
\tag{3.13}
\end{equation}
Apply Lemma 3.1 to the first term in (3.10). To treat the second term, define
\[
\widetilde\phi(y)=\phi(Qy).
\]
Since reflection is a Euclidean isometry, $\phi\in\mathcal H_0^s(D)$ directly implies $\widetilde\phi\in\mathcal H_0^s(Q(D))$. By the assumption $Q(D)\subset\Omega$ and the definition by zero extension, $\mathcal H_0^s(Q(D))\subset\mathcal H_0^s(\Omega)$. Reflection also preserves Lebesgue measure and the fractional energy. Thus, after the change of variables $y=Qx$,
\begin{align*}
\int_A\delta(Qx)^{-b}\phi(x)^2\,dx
&=\int_{Q(A)}\delta(y)^{-b}\widetilde\phi(y)^2\,dy\\
&\le C|Q(A)|^{sp/N}\mathcal E_s(\widetilde\phi,\widetilde\phi)\\
&=C|A|^{sp/N}\mathcal E_s(\phi,\phi).
\end{align*}
Combining this with the first term gives
\[
\mathcal E_s(\phi,\phi)\le CC_0|A|^{sp/N}\mathcal E_s(\phi,\phi).
\]
If $\varepsilon_0$ is chosen so that $CC_0\varepsilon_0^{sp/N}<1$, then $\mathcal E_s(\phi,\phi)=0$, and hence $\phi=0$. \hfill$\square$

\medskip
\noindent\textbf{Proposition 3.3 (Initiation and continuation of moving planes).}
For a solution of (1.4), the moving plane in any direction can be initiated from outside the domain and can be moved continuously up to the first geometric critical position; before that position, the reflected difference is nonnegative.

\smallskip
\noindent\textit{Proof.}
Fix a unit vector $e$, and set
\[
\lambda^+=\sup_{x\in\Omega}x\cdot e,
\qquad H_\lambda=\{x:x\cdot e>\lambda\},
\qquad \Sigma_\lambda=\Omega\cap H_\lambda.
\]
Let $Q_\lambda$ be reflection across $\partial H_\lambda$. The first geometric critical value is defined by
\[
\lambda_*=\inf\bigl\{\lambda<\lambda^+:Q_\mu(\Sigma_\mu)\subset\Omega
\text{ for every }\mu\in(\lambda,\lambda^+)\bigr\}.
\]
For $\lambda\in(\lambda_*,\lambda^+)$, define throughout $H_\lambda$
\[
w_\lambda(x)=u(Q_\lambda x)-u(x).
\]
If $x\in H_\lambda\setminus\Sigma_\lambda$, then $u(x)=0$, and hence $w_\lambda(x)=u(Q_\lambda x)\ge0$. Thus the negative set is entirely contained in $\Sigma_\lambda$.

For $x\in\Sigma_\lambda$, both $x$ and $Q_\lambda x$ belong to $\Omega$, so we may define
\[
c_\lambda(x)=\Lambda q\int_0^1
\bigl((1-t)u(x)+tu(Q_\lambda x)\bigr)^{q-1}\,dt.
\]
The fundamental theorem of calculus gives
\[
(-\Delta)^s w_\lambda=c_\lambda(x)w_\lambda
\qquad\text{in }\Sigma_\lambda.
\]
At points where $w_\lambda<0$, one has $0<u(Q_\lambda x)<u(x)$. By the concavity of $r\mapsto r^q$,
\[
c_\lambda(x)\le\Lambda q\,u(Q_\lambda x)^{q-1}
\le C\delta(Q_\lambda x)^{-s(1-q)},
\]
where the last step uses (3.2). Hence (3.10) holds.

As $\lambda\uparrow\lambda^+$, one has $|\Sigma_\lambda|\to0$. Choose $|\Sigma_\lambda|<\varepsilon_0$; Lemma 3.2 then gives $w_\lambda\ge0$, so the moving plane can be initiated.

We next prove continuation by an open-and-closed argument. Let $\lambda_0>\lambda_*$ and suppose that $w_\lambda\ge0$ for $\lambda\in[\lambda_0,\lambda^+)$. If $w_{\lambda_0}\not\equiv0$, then in $\Sigma_{\lambda_0}$,
\[
(-\Delta)^s w_{\lambda_0}
=\Lambda\bigl(u(Q_{\lambda_0}x)^q-u(x)^q\bigr)\ge0.
\]
Applying the antisymmetric strong maximum principle here with $c=0$ gives $w_{\lambda_0}>0$ in $\Sigma_{\lambda_0}$. By interior regularity, $w_{\lambda_0}$ is continuous. First choose a compact set $K\Subset\Sigma_{\lambda_0}$ such that
\[
|\Sigma_{\lambda_0}\setminus K|<\frac{\varepsilon_0}{4},
\]
and then use strict positivity on $K$ to choose $\eta>0$ such that $w_{\lambda_0}\ge2\eta$ on $K$. The uniform continuity of $u$ and the convergence $Q_\lambda\to Q_{\lambda_0}$ imply that, when $\lambda<\lambda_0$ is sufficiently close to $\lambda_0$,
\[
K\subset\Sigma_\lambda,
\qquad w_\lambda\ge\eta\quad\text{on }K.
\]
Moreover, since $\partial\Omega$ has measure zero, the cap depends continuously on the parameter in measure; after shrinking the parameter interval, one may also arrange that
\[
|\Sigma_\lambda\mathbin{\triangle}\Sigma_{\lambda_0}|<\frac{\varepsilon_0}{2}.
\]
Thus the possible negative set $A_\lambda$ is contained in $\Sigma_\lambda\setminus K$, and
\[
|A_\lambda|
\le|\Sigma_{\lambda_0}\setminus K|
+|\Sigma_\lambda\mathbin{\triangle}\Sigma_{\lambda_0}|
<\varepsilon_0.
\]
Lemma 3.2 again gives $w_\lambda\ge0$.

It remains to rule out the premature occurrence of symmetry when $\lambda_0>\lambda_*$. If $w_{\lambda_0}\equiv0$, then $u$ and $\Omega$ are symmetric with respect to $\partial H_{\lambda_0}$. Set
\[
\lambda^-=\inf_{x\in\Omega}x\cdot e.
\]
Symmetry gives $\lambda^-=2\lambda_0-\lambda^+$. If $\lambda_0>\lambda_*$, choose $\mu\in(\lambda_*,\lambda_0)$ and then choose $x_n\in\Omega$ such that $x_n\cdot e\to\lambda^+$. For all sufficiently large $n$, one has $x_n\in\Sigma_\mu$, but
\[
Q_\mu x_n\cdot e=2\mu-x_n\cdot e
\longrightarrow2\mu-\lambda^+
<2\lambda_0-\lambda^+=\lambda^-.
\]
Thus $Q_\mu x_n\notin\Omega$, contradicting $Q_\mu(\Sigma_\mu)\subset\Omega$. Therefore $w_{\lambda_0}\equiv0$ necessarily implies $\lambda_0=\lambda_*$; when $\lambda_0>\lambda_*$, only strict positivity is possible, and the preceding narrow-domain argument permits the plane to continue moving inward. Hence the moving-plane process continues precisely up to the first geometric critical position. \hfill$\square$

\section{Curved-boundary powers and resonant logarithmic terms}

This section establishes the local formulas needed for the finite expansion below. Let $d$ be a $C^3$ defining function: in a tubular neighborhood of the boundary, $d$ agrees with the signed distance (positive inside the domain), and it is extended smoothly away from the boundary. We write $d_+=\max\{d,0\}$. Let $\chi$ be a local cutoff that equals $1$ near the boundary point under consideration.

\subsection{One-dimensional coefficients}

For $s<\alpha<1$, define
\begin{equation}
  \kappa_s(\alpha)
  =\frac{\Gamma(s+\alpha+1)\sin(\pi\alpha)}
  {\Gamma(\alpha-s+1)\sin(\pi(\alpha-s))}.
  \tag{4.1}\label{eq:4.1}
\end{equation}
A one-dimensional distributional computation gives
\begin{equation}
  (-\partial_t^2)^s\bigl(t_+^{s+\alpha}\bigr)
  =\kappa_s(\alpha)t^{\alpha-s},
  \qquad t>0.
  \tag{4.2}\label{eq:4.2}
\end{equation}
The formula may first be computed by a Beta integral in a strip of complex parameters where the integral converges absolutely, and then analytically continued. Since $\alpha-s\in(0,1-s)$, the denominator has no zeros; hence $\kappa_s(\alpha)\ne0$. More precisely, the unique zero of $\kappa_s$ in $(s,1]$ is $\alpha=1$, this zero is simple, and
\begin{equation}
  \kappa_s'(1)
  =-\frac{\pi\Gamma(s+2)}{\Gamma(2-s)\sin(\pi s)}\ne0.
  \tag{4.3}\label{eq:4.3}
\end{equation}
Differentiating \eqref{eq:4.2} with respect to $\alpha$ yields
\begin{equation}
  (-\partial_t^2)^s\bigl(t_+^{s+1}\log t_+\bigr)
  =\kappa_s'(1)t^{1-s},
  \qquad t>0.
  \tag{4.4}\label{eq:4.4}
\end{equation}
Here we adopt the convention that $t_+^{s+1}\log t_+=0$ for $t\le0$, and the identity is understood in the sense of distributions.

\subsection{Finite-regularity changes of coordinates}

We use the finite-regularity pseudodifferential-operator change-of-coordinates theorem of Abels--Grubb~[5, Theorems~3.4, 3.8, 3.9 and Corollary~3.5], whose specialized form needed here is the following: if $\Phi$ is a $C^3$ diffeomorphism and $P$ is a classical even operator of order $2s$, then, after removal of the low-frequency part,
\begin{equation}
  \Phi^*P(\Phi^{-1})^*=Q_0+Q_1+R_1+S.
  \tag{4.5}\label{eq:4.5}
\end{equation}
Here the order of the diagonal expansion in Theorem~3.8 is taken to be $1$: the principal symbol of $Q_0$ is the transformed principal symbol of order $2s$, $Q_1$ is the first diagonal term, and $R_1$ is the double-symbol remainder vanishing on the diagonal; the latter two have order at most $2s-1$. Theorem~3.4 and Corollary~3.5 provide the required local Sobolev mappings within their restrictions on the exponents; below we verify these restrictions term by term for the specific target exponents. The coordinate remainder $S$ in Theorem~3.9 maps $L^r$ into every $H_r^\sigma$ with $\sigma<2-2s$. The low-frequency part, when applied to a compactly supported distribution, is locally smooth. Below, $C_*^\theta=B_{\infty,\infty}^\theta$ denotes the H\"older--Zygmund space; when $\theta\notin\mathbb N$, it is the usual H\"older space $C^\theta$.

\begin{lemma}[Weak curved-boundary power formula]\label{lem:4.1}
Let $s<\alpha<1$. Then, in every boundary coordinate patch,
\begin{equation}
  (-\Delta)^s\bigl(\chi d_+^{s+\alpha}\bigr)
  =\kappa_s(\alpha)d^{\alpha-s}+G_\alpha,
  \tag{4.6}\label{eq:4.6}
\end{equation}
where
\begin{equation}
  G_\alpha\in C_*^\theta
  \qquad\text{for every}\qquad
  \theta<\min\{1,1+\alpha-s,2-2s\}.
  \tag{4.7}\label{eq:4.7}
\end{equation}
At $\alpha=1$,
\begin{equation}
  (-\Delta)^s\bigl(\chi d_+^{s+1}\log d_+\bigr)
  =\kappa_s'(1)d^{1-s}+G_{\log},
  \tag{4.8}\label{eq:4.8}
\end{equation}
and $G_{\log}\in C_*^\theta$ for every $\theta<\min\{1,2-2s\}$.
\end{lemma}

\begin{proof}
Fix a boundary point and translate and rotate the coordinates so that the local region is written as
\[
  \Omega=\{(z,x_N):x_N>\varphi(z)\},
  \qquad \varphi\in C^3.
\]
Use the graph coordinates
\begin{equation}
  \Phi(z,t)=(z,\varphi(z)+t).
  \tag{4.9}\label{eq:4.9}
\end{equation}
After multiplication by a tangential cutoff on a smaller coordinate patch, $\varphi$ can be extended to the whole space in such a way that $\Phi$ extends to a global $C^3$ diffeomorphism equal to the identity outside a large ball; the following local computation is unaffected by this extension. Thus $D\Phi\in C^2$, and the regularity parameter in the Abels--Grubb theorem can be taken to be $\tau=2>2s$. Set
\[
  D(z,t)=d(\Phi(z,t)).
\]
Since $D(z,0)=0$, Taylor's formula gives
\begin{equation}
  D(z,t)=tA(z,t),
  \qquad A\in C^2,
  \qquad a(z):=A(z,0)=\frac{1}{\sqrt{1+|\nabla\varphi(z)|^2}}.
  \tag{4.10}\label{eq:4.10}
\end{equation}

First decompose $|\xi|^{2s}$ into high- and low-frequency parts. The low-frequency term and the cutoff terms whose supports are separated from the observation point are locally smooth. Apply \eqref{eq:4.5} to the high-frequency part. Its principal symbol is
\[
  q_0(z,t,\xi)=\bigl|(D\Phi(z,t))^{-T}\xi\bigr|^{2s}.
\]
Because the differential of the graph coordinates \eqref{eq:4.9} is independent of $t$, at purely normal frequencies one has, for every $t$,
\begin{equation}
  q_0(z,t,0,\tau)=a(z)^{-2s}|\tau|^{2s}.
  \tag{4.11}\label{eq:4.11}
\end{equation}

Let $\mu=s+\alpha$ and $f(t)=\chi(t)t_+^\mu$, and denote multiplication by $A^\mu$ by $M$. By \eqref{eq:4.10}, the principal term is $Q_0(Mf)$. Decompose the commutator as
\begin{equation}
  Q_0(Mf)=M Q_0f+[Q_0,M]f.
  \tag{4.12}\label{eq:4.12}
\end{equation}
Here one cannot regard $M=A^\mu\in C^2$ as a smooth multiplier and directly invoke the smooth commutator formula. Write
\[
  M(y)-M(x)
  =(y-x)\cdot\int_0^1\nabla M\bigl(x+\theta(y-x)\bigr)\,d\theta.
\]
Integrating once by parts in $\xi$ after applying the factor $y-x$ to the oscillatory exponential shows that $[Q_0,M]$ is an $x,y$-form double-symbol operator: by~[5, Lemma~3.2], its amplitude is a finite linear combination of $\partial_\xi q_0$ and the integral coefficient in the preceding display, and therefore has order $2s-1$ and coefficients of regularity $C^1$. Since $|2s-1|<1$, [5, Theorem~3.4] gives exactly the local Sobolev mapping stated above. The commutator with the tangential cutoff is treated by the same computation; the part whose support is separated from the observation point is smooth by pseudolocality. They are therefore included, together with $Q_1$, $R_1$, and $S$, in the lower-order remainder. Since $f$ is independent of the tangential variables, \eqref{eq:4.11} and the one-dimensional formula \eqref{eq:4.2} give
\begin{equation}
  M Q_0f
  =\kappa_s(\alpha)A^\mu a^{-2s}t^{\mu-2s}
  +\text{lower-order terms}.
  \tag{4.13}\label{eq:4.13}
\end{equation}
On the other hand,
\[
  d^{\alpha-s}=D^{\mu-2s}=A^{\mu-2s}t^{\mu-2s}.
\]
Moreover, since $A(z,t)=a(z)+O(t)$,
\[
  A^\mu a^{-2s}-A^{\mu-2s}
  =A^{\mu-2s}\left[\left(\frac{A}{a}\right)^{2s}-1\right]
  =O(t).
\]
Consequently, the difference between \eqref{eq:4.13} and $\kappa_s(\alpha)d^{\alpha-s}$ is
\begin{equation}
  O\bigl(t^{1+\alpha-s}\bigr).
  \tag{4.14}\label{eq:4.14}
\end{equation}

It remains only to verify the function-space exponents. With a tangential cutoff, $t_+^\mu$ belongs to
\[
  H_r^\gamma
  \qquad\text{for every }\gamma<\mu+\frac1r.
\]
Now fix
\[
  \theta<\min\{1,1+\alpha-s,2-2s\}.
\]
Choose $r$ sufficiently large and then choose $\sigma_0$ such that
\[
  \theta+\frac Nr<\sigma_0
  <\min\left\{1,2-2s,1+\alpha-s+\frac1r\right\}.
\]
Then
\[
  |\sigma_0|<1,
  \qquad |\sigma_0+2s-1|<1,
  \qquad \sigma_0+2s-1<\mu+\frac1r.
\]
Theorem~3.4 therefore maps the commutator and the other terms of order $2s-1$ into $H_r^{\sigma_0}$, Corollary~3.5 gives the same conclusion for the remainder vanishing on the diagonal, and $\sigma_0<2-2s$ also satisfies the upper bound for the coordinate remainder. Finally, the embedding
\[
  H_r^{\sigma_0}\hookrightarrow C_*^{\sigma_0-N/r}
\]
gives \eqref{eq:4.7}. It also controls \eqref{eq:4.14}.

The logarithmic case is handled using
\begin{equation}
  D^{s+1}\log D
  =A^{s+1}t^{s+1}(\log t+\log A).
  \tag{4.15}\label{eq:4.15}
\end{equation}
For the first term, differentiate the flat one-dimensional analytic family at $\alpha=1$ and use \eqref{eq:4.3}; the leading coefficient of the ordinary power in the second term vanishes because $\kappa_s(1)=0$. The remaining commutators and coordinate remainders are still operators of order $2s-1$, so the same mapping estimates apply directly, without differentiating a parameter through the curved-boundary principal-value integral. This proves \eqref{eq:4.8}.
\end{proof}

\begin{lemma}[Base boundary power]\label{lem:4.2}
If $\Omega$ is $C^3$, then, for all sufficiently small $\varepsilon>0$,
\begin{equation}
  (-\Delta)^s\bigl(\chi d_+^s\bigr)\in C^{1+\varepsilon-s}
  \tag{4.16}\label{eq:4.16}
\end{equation}
in a neighborhood of the boundary.
\end{lemma}

\begin{proof}
Abatangelo--Ros-Oton~[4, Corollary~2.3] proved that if the boundary is $C^\beta$, with $\beta>1+s$ and $\beta-s\notin\mathbb N$, then
\[
  (-\Delta)^s\bigl(\chi d_+^s\bigr)\in C^{\beta-1-s}.
\]
Take $\beta<3$ sufficiently close to $3$ while avoiding the discrete exceptional values above, and then take $0<\varepsilon<\beta-2$. This gives $1+\varepsilon-s<\beta-1-s$.
\end{proof}

\section{Finite expansion of the boundary quotient}

\subsection{Higher-order boundary Schauder estimates}

We use~[4, Theorem~1.4 and Proposition~4.1] below. The specialized form needed here is as follows: if $\beta>s$, neither $\beta$ nor $\beta\pm s$ is an integer, $\Omega$ is $C^{\beta+1}$, $v\in L^\infty(\mathbb R^N)$, and
\[
  (-\Delta)^s v=f\quad\text{in }\Omega,
  \qquad v=0\quad\text{in }\Omega^c,
  \qquad f\in C^{\beta-s},
\]
then
\begin{equation}
  \frac{v}{d^s}\in C^\beta(\overline\Omega).
  \tag{5.1}\label{eq:5.1}
\end{equation}
Moreover, after subtracting the boundary Taylor polynomial multiplied by $d^s$, the remainder satisfies, on every interior ball, scale-invariant H\"older estimates consistent with its order of vanishing. To avoid ambiguity, we make the following definition. Let $0<\beta<1$, assume $s+\beta\notin\mathbb N$, let $R=0$ on $\partial\Omega$, and set
\[
  v=d^sR,
  \qquad \gamma=s+\beta,
  \qquad m=\lfloor\gamma\rfloor,
  \qquad \vartheta=\gamma-m\in(0,1).
\]
If
\begin{equation}
  |v(x)|\le C d(x)^\gamma
  \tag{5.2}\label{eq:5.2}
\end{equation}
and, for every interior point satisfying $d(x_0)=2r$,
\begin{equation}
  \sum_{k=0}^m r^{k-\gamma}
  \|D^k v\|_{L^\infty(B_r(x_0))}
  +[D^m v]_{C^\vartheta(B_r(x_0))}\le C,
  \tag{5.3}\label{eq:5.3}
\end{equation}
then $R$ is called a weighted remainder of order $\beta$. This terminology refers only to the two estimates \eqref{eq:5.2}--\eqref{eq:5.3}. We denote the collection of such functions by $\mathcal R_\beta$. Since $v=d^sR$, $d\asymp r$, and the Leibniz formula holds, \eqref{eq:5.3} equivalently also gives the interior-ball scale estimates
\begin{equation}
  \|D^kR\|_{L^\infty(B_r(x_0))}\le Cr^{\beta-k}
  \qquad(k\le m),
  \tag{5.4}\label{eq:5.4}
\end{equation}
together with the corresponding top-order H\"older seminorm estimate; this equivalent form will be used below when verifying products. Abatangelo--Ros-Oton~[4, Proposition~4.1, in particular (44)] shows that the quotient obtained after subtracting the boundary Taylor polynomial in \eqref{eq:5.1} satisfies precisely these estimates. All the exponents used below can be perturbed arbitrarily slightly so as to avoid the finitely many exceptional integer values.

The next lemma explains why multiplication of a vanishing boundary remainder by $d^a$ increases the exponent; this is not the ordinary H\"older multiplication rule.

\begin{lemma}[Multiplication of a vanishing remainder]\label{lem:5.1}
Let $0<a<s$ and $0<\beta<1$, and let $R$ be a weighted remainder of order $\beta$ in the sense of \eqref{eq:5.2}--\eqref{eq:5.3}. Then
\begin{equation}
  d^aR\in C^\zeta(\overline\Omega)
  \qquad\text{for every}\qquad
  0<\zeta<\min\{1,a+\beta\}.
  \tag{5.5}\label{eq:5.5}
\end{equation}
If $a+\beta<1$ and $s+\beta$ is not an integer, one may also take $\zeta=a+\beta$.
\end{lemma}

\begin{proof}
Take $x,y\in\Omega$, set $h=|x-y|$, and assume without loss of generality that $d(x)\le d(y)$. If $h\ge d(x)/2$, then \eqref{eq:5.2} gives $|R(x)|\le Cd(x)^\beta$. Since $d$ is $1$-Lipschitz, $d(y)\le d(x)+h\le3h$, and therefore
\[
  |d(x)^aR(x)-d(y)^aR(y)|\le Ch^{a+\beta}.
\]

Now suppose that $h<d(x)/2$. Then $t:=d(x)\asymp d(y)$. The interior-ball scale estimates and interpolation give
\begin{equation}
  |R(x)-R(y)|
  \le C\bigl(t^{-s}h^{s+\beta}+t^{\beta-1}h\bigr),
  \tag{5.6}\label{eq:5.6}
\end{equation}
where the first term may be omitted when $s+\beta>1$. Hence
\begin{align*}
  |d(x)^aR(x)-d(y)^aR(y)|
  &\le d(x)^a|R(x)-R(y)|
     +|R(y)|\,|d(x)^a-d(y)^a|\\
  &\le C\bigl(t^{a-s}h^{s+\beta}+t^{a+\beta-1}h\bigr).
\end{align*}
Since $a<s$ and $h<t$,
\[
  t^{a-s}h^{s+\beta}
  =h^{a+\beta}\left(\frac ht\right)^{s-a}
  \le h^{a+\beta}.
\]
If $a+\beta<1$, then
\[
  t^{a+\beta-1}h\le h^{a+\beta}.
\]
This proves the endpoint $\zeta=a+\beta<1$. If $a+\beta\ge1$, then $t^{a+\beta-1}h\le Ch$; for every $\zeta<1$, when $h<1$ one has $h\le h^\zeta$, while the preceding term $h^{a+\beta}\le h^\zeta$. Combining the two distance regimes proves \eqref{eq:5.5}.
\end{proof}

\begin{lemma}[Stability of weighted remainders under nonlinear composition]\label{lem:5.2}
Let $j\ge0$ and $\rho>s$, and assume that $0<\beta<(j+1)\rho$ and $s+\beta\notin\mathbb N$. Let
\[
  P_j(z)=C_0+\sum_{\ell=1}^j A_\ell z^\ell,
  \qquad C_0>0,
\]
and let $R\in\mathcal R_\beta$. Then, in a sufficiently small boundary neighborhood,
\begin{equation}
  \bigl(P_j(d^\rho)+R\bigr)^q
  =\sum_{\ell=0}^j B_\ell d^{\ell\rho}+\widetilde R,
  \qquad
  B_\ell=[z^\ell]P_j(z)^q,
  \tag{5.7}\label{eq:5.7}
\end{equation}
where $\widetilde R\in\mathcal R_\beta$.
\end{lemma}

\begin{proof}
Let $F(t)=t^q$. Since $R=O(d^\beta)$ and $P_j(d^\rho)\to C_0$, after shrinking the neighborhood, both $P_j(d^\rho)$ and $P_j(d^\rho)+R$ lie in the same compact subinterval of $(0,\infty)$. Split the remainder as
\[
  E_1=F\bigl(P_j(d^\rho)+R\bigr)-F\bigl(P_j(d^\rho)\bigr),
  \qquad
  E_2=F\bigl(P_j(d^\rho)\bigr)-\sum_{\ell=0}^j B_\ell d^{\ell\rho}.
\]
The function $z\mapsto F(P_j(z))$ is analytic near $z=0$, so there is a smooth function $Q_j$ such that
\[
  E_2=d^{(j+1)\rho}Q_j(d^\rho).
\]
Because $(j+1)\rho>\beta$, differentiating the preceding expression term by term on every interior ball with $d(x_0)=2r$ gives the estimates with exponent $\beta$ in \eqref{eq:5.2}--\eqref{eq:5.3}; hence $E_2\in\mathcal R_\beta$.

The integral form of Taylor's formula gives
\begin{equation}
  E_1=R C_R,
  \qquad
  C_R=\int_0^1 F'\bigl(P_j(d^\rho)+tR\bigr)\,dt.
  \tag{5.8}\label{eq:5.8}
\end{equation}
Write $C_R=F'(C_0)+\widehat C_R$. On an interior ball at depth $r$, \eqref{eq:5.4} and
\[
  \|D^k d^{\ell\rho}\|_\infty\le Cr^{\ell\rho-k}
\]
give
\[
  \|D^k\widehat C_R\|_\infty
  \le Cr^{\min\{\rho,\beta\}-k},
\]
together with the corresponding scale H\"older estimates; this follows directly by applying the chain rule to \eqref{eq:5.8}. Thus
\[
  d^sE_1=F'(C_0)d^sR+\widehat C_Rd^sR.
\]
The first term satisfies the original estimates with exponent $s+\beta$, while the order of vanishing and the interior-ball scale order of the second term are at least $s+\beta+\min\{\rho,\beta\}$, and hence it also satisfies the weaker estimates with exponent $s+\beta$. Therefore $E_1\in\mathcal R_\beta$. Adding the two parts proves the conclusion.
\end{proof}

\subsection{Finite recursion}

\begin{theorem}[Boundary quotient expansion]\label{thm:5.3}
Suppose that $u$ satisfies \textup{(1.4)}, and set
\[
  \rho=s(1+q),
  \qquad a=sq=\rho-s.
\]
Then there exist $\varepsilon>0$ and $H\in C^{1,\varepsilon}(\overline\Omega)$ such that the following assertions hold.

If there is no positive integer $M$ such that $M\rho=1$, then
\begin{equation}
  \frac{u}{d^s}
  =H+\sum_{\substack{k\ge1\\k\rho<1}}A_kd^{k\rho}.
  \tag{5.9}\label{eq:5.9}
\end{equation}
If there is a positive integer $M$ such that $M\rho=1$, then
\begin{equation}
  \frac{u}{d^s}
  =H+\sum_{\substack{k\ge1\\k\rho<1}}A_kd^{k\rho}
  +A_*d\log d.
  \tag{5.10}\label{eq:5.10}
\end{equation}
All the $A_k$ and $A_*$ are constants independent of the boundary point, and $H=C_0$ on $\partial\Omega$.
\end{theorem}

\begin{proof}
If $\rho>1$, the basic boundary regularity $u\in C^s(\mathbb R^N)$ and $0<q<1$ give
\[
  u^q\in C^{sq}=C^{\rho-s}.
\]
Choose $0<\varepsilon<\rho-1$ while avoiding the exceptional integer values in the Schauder theorem, and take $\beta=1+\varepsilon$ in \eqref{eq:5.1}. This yields $u/d^s\in C^{1,\varepsilon}$. It therefore remains only to consider $\rho\le1$.

All computations below are performed in a fixed tubular neighborhood of the boundary. Every power is multiplied by a cutoff $\chi$ equal to $1$ in a smaller tubular neighborhood; the terms generated by the cutoff away from the boundary are smooth and will eventually all be absorbed into $H$. Write
\[
  G_0=(-\Delta)^s(\chi d_+^s),
  \qquad
  G_\alpha=(-\Delta)^s(\chi d_+^{s+\alpha})
  -\kappa_s(\alpha)d^{\alpha-s}.
\]
Fix a sufficiently small $\eta>0$ such that the following finitely many exponents avoid the exceptional integer values in the Schauder theorem and
\begin{equation}
  0<\eta<\min\{\rho-s,\rho,1-s\}.
  \tag{5.11}\label{eq:5.11}
\end{equation}
If there is no positive integer $M$ such that $M\rho=1$, also let $K\rho<1<(K+1)\rho$ and additionally require
\begin{equation}
  \eta<(K+1)\rho-1.
  \tag{5.12}\label{eq:5.12}
\end{equation}
Because only finitely many exponents occur, one may simultaneously require that $\beta_j=(j+1)\rho-\eta$, $\beta_j\pm s$, and $s+\beta_j$ all avoid integer values. Set $\psi=u/d^s$. The basic boundary regularity gives $u\in C^s(\mathbb R^N)$; since $|r^q-t^q|\le|r-t|^q$,
\[
  u^q\in C^{sq}=C^{\rho-s}.
\]
First let $v_0=u-C_0\chi d_+^s$. Its right-hand side
\[
  \Lambda u^q-C_0(-\Delta)^s(\chi d_+^s)
\]
belongs to $C^{\rho-s-\eta}$: the first term even belongs to $C^{\rho-s}$, and the second is smoother by Lemma~\ref{lem:4.2}. Taking $\beta_0=\rho-\eta>s$ in \eqref{eq:5.1} gives
\begin{equation}
  \psi=C_0+R_0,
  \tag{5.13}\label{eq:5.13}
\end{equation}
where $R_0|_{\partial\Omega}=0$. The interior-ball scale estimates in~[4, Proposition~4.1, formula~(44)] then give
\begin{equation}
  R_0\in\mathcal R_{\rho-\eta}.
  \tag{5.14}\label{eq:5.14}
\end{equation}
Here $\rho-\eta<1$; even when $\rho=1$, the strict loss $\eta$ ensures that the definition applies.

We now perform a finite induction. Suppose that, for some $j\ge0$, we have obtained
\begin{equation}
  \psi=C_0+\sum_{\ell=1}^jA_\ell d^{\ell\rho}+R_j,
  \tag{5.15}\label{eq:5.15}
\end{equation}
where
\begin{equation}
  R_j\in\mathcal R_{\beta_j},
  \qquad \beta_j=(j+1)\rho-\eta<1.
  \tag{5.16}\label{eq:5.16}
\end{equation}
To check the nonlinear remainder term by term, set
\[
  P_j=C_0+\sum_{\ell=1}^jA_\ell d^{\ell\rho},
  \qquad F(t)=t^q.
\]
Regard $P_j$ as a polynomial in $z=d^\rho$, and define
\begin{equation}
  B_k=[z^k]\left(C_0+\sum_{\ell=1}^jA_\ell z^\ell\right)^q,
  \qquad 0\le k\le j.
  \tag{5.17}\label{eq:5.17}
\end{equation}
Since $\beta_j=(j+1)\rho-\eta<(j+1)\rho$, Lemma~\ref{lem:5.2} gives, term by term,
\begin{equation}
  \psi^q=\sum_{\ell=0}^jB_\ell d^{\ell\rho}+\widetilde R_j,
  \tag{5.18}\label{eq:5.18}
\end{equation}
where
\begin{equation}
  \widetilde R_j\in\mathcal R_{\beta_j}.
  \tag{5.19}\label{eq:5.19}
\end{equation}
Notice that the newly added term $A_jz^j$ does not change the coefficients of the lower powers, so the $B_0,\ldots,B_j$ defined at successive steps are compatible. The right-hand side of the equation is
\begin{equation}
  \Lambda u^q=\Lambda d^{\rho-s}\psi^q.
  \tag{5.20}\label{eq:5.20}
\end{equation}
Thus the $j$th as-yet-uncancelled purely normal term in \eqref{eq:5.20} is
\[
  \Lambda B_jd^{(j+1)\rho-s}.
\]

If $(j+1)\rho<1$, note that $(j+1)\rho\ge\rho>s$. Since $\kappa_s((j+1)\rho)\ne0$ by Lemma~\ref{lem:4.1}, define
\begin{equation}
  A_{j+1}=\frac{\Lambda B_j}{\kappa_s((j+1)\rho)}.
  \tag{5.21}\label{eq:5.21}
\end{equation}
After subtracting $A_{j+1}\chi d_+^{s+(j+1)\rho}$ from $u$, the principal term in \eqref{eq:4.6} cancels exactly $\Lambda B_jd^{(j+1)\rho-s}$. More precisely, define
\[
  v_{j+1}=u-C_0\chi d_+^s
  -\sum_{\ell=1}^{j+1}A_\ell\chi d_+^{s+\ell\rho}.
\]
The relations $A_\ell\kappa_s(\ell\rho)=\Lambda B_{\ell-1}$ from the preceding steps give the full remainder equation
\begin{equation}
  (-\Delta)^sv_{j+1}
  =\Lambda d^{\rho-s}\widetilde R_j-C_0G_0
  -\sum_{\ell=1}^{j+1}A_\ell G_{\ell\rho}.
  \tag{5.22}\label{eq:5.22}
\end{equation}
By \eqref{eq:5.19} and Lemma~\ref{lem:5.1},
\begin{equation}
  d^{\rho-s}\widetilde R_j
  \in C^{(j+2)\rho-s-\eta}
  \tag{5.23}\label{eq:5.23}
\end{equation}
provided that $(j+2)\rho\le1$; here
\[
  (\rho-s)+\beta_j=(j+2)\rho-s-\eta<1,
\]
and \eqref{eq:5.11}, together with the subsequent choice of noninteger exponents, has been used to satisfy the endpoint condition. The curvature remainder belongs to every
\[
  C_*^\theta,
  \qquad
  \theta<\min\{1,1+(j+1)\rho-s,2-2s\}.
\]
When $(j+2)\rho\le1$, the right-hand-side exponent required for the next Schauder estimate is
\[
  (j+2)\rho-s-\eta;
\]
it is strictly below all three bounds above, because
\[
  1+(j+1)\rho-s-\bigl((j+2)\rho-s-\eta\bigr)
  =1-\rho+\eta>0,
\]
\[
  1-\bigl((j+2)\rho-s-\eta\bigr)>0,
\]
and
\[
  2-2s-\bigl((j+2)\rho-s-\eta\bigr)>1-s>0.
\]
For every $1\le\ell\le j+1$, the third regularity restriction for an earlier curvature remainder is also satisfied, since
\[
  1+\ell\rho-s-\bigl((j+2)\rho-s-\eta\bigr)
  =1-(j+2-\ell)\rho+\eta
  \ge1-(j+1)\rho+\eta>0,
\]
where the last step uses $(j+2)\rho\le1$. The base term is controlled by Lemma~\ref{lem:4.2}. Therefore, applying \eqref{eq:5.1} to \eqref{eq:5.22} with
\[
  \beta_{j+1}=(j+2)\rho-\eta
\]
and then using~[4, Proposition~4.1, (44)] gives
\[
  R_{j+1}\in\mathcal R_{\beta_{j+1}}.
\]
This proves the induction step. Because the same $\eta$ is used at every step, the loss does not accumulate.

In the nonresonant case, let $K\ge0$ be the unique integer such that
\[
  K\rho<1<(K+1)\rho.
\]
The present discussion concerns $\rho\le1$, so here $K\ge1$. The induction first gives
\[
  R_{K-1}\in\mathcal R_{K\rho-\eta}.
\]
The Taylor expansion above with $j=K-1$ then defines $B_{K-1}$ and $A_K$ and cancels the last purely normal term of order below one. At this point,
\[
  d^{\rho-s}\widetilde R_{K-1}\in C^\zeta
  \qquad\text{for every}\qquad
  \zeta<\min\{1,(K+1)\rho-s-\eta\}.
\]
Choose
\begin{equation}
  0<\varepsilon<\min\{s,1-s,\rho,(K+1)\rho-1-\eta\},
  \tag{5.24}\label{eq:5.24}
\end{equation}
while avoiding the exceptional integer values. Then
\[
  1+\varepsilon-s
  <\min\{1,(K+1)\rho-s-\eta\}.
\]
At the same time, $\varepsilon<1-s$ ensures that $1+\varepsilon-s<2-2s$, so all the curvature remainders and the base term also belong to $C^{1+\varepsilon-s}$. Finally, apply \eqref{eq:5.1}, with $\beta=1+\varepsilon$, to
\[
  u-C_0\chi d_+^s
  -\sum_{k=1}^K A_k\chi d_+^{s+k\rho}.
\]
Its boundary quotient belongs to $C^{1,\varepsilon}$. Absorbing into this quotient the smooth terms generated by the cutoff in the interior gives \eqref{eq:5.9}.

If $M\rho=1$, then for $M\ge2$ first recurse up to $A_{M-1}$; after the last ordinary power has been cancelled, the equality case $(j+2)\rho=1$ in the induction above gives
\[
  R_{M-1}\in\mathcal R_{1-\eta}.
\]
When $M=1$, this is exactly the initial conclusion $R_0\in\mathcal R_{1-\eta}$. Thus the two cases can both be written as
\[
  \psi=C_0+\sum_{k=1}^{M-1}A_kd^{k\rho}+R_{M-1},
  \qquad R_{M-1}\in\mathcal R_{1-\eta}.
\]
Applying Lemma~\ref{lem:5.2} once more gives
\[
  \psi^q=\sum_{\ell=0}^{M-1}B_\ell d^{\ell\rho}
  +\widetilde R_{M-1},
  \qquad \widetilde R_{M-1}\in\mathcal R_{1-\eta}.
\]
The only uncancelled critical purely normal term is
\[
  \Lambda B_{M-1}d^{1-s}.
\]
Since $\kappa_s(1)=0$, the ordinary power $d^{s+1}$ cannot cancel it. By \eqref{eq:4.8}, set
\begin{equation}
  A_*=\frac{\Lambda B_{M-1}}{\kappa_s'(1)}.
  \tag{5.25}\label{eq:5.25}
\end{equation}
After subtracting $A_*\chi d_+^{s+1}\log d_+$, this term is cancelled. By Lemma~\ref{lem:5.1}, the remaining nonlinear remainder satisfies
\[
  d^{\rho-s}\widetilde R_{M-1}\in C^\zeta
  \qquad\text{for every}\qquad
  \zeta<\min\{1,1+\rho-s-\eta\}.
\]
Choose
\begin{equation}
  0<\varepsilon<\min\{s,1-s,\rho-\eta\},
  \tag{5.26}\label{eq:5.26}
\end{equation}
while avoiding the exceptional integer values. Then
\[
  1+\varepsilon-s
  <\min\{1,1+\rho-s-\eta,2-2s\}.
\]
Hence the nonlinear remainder above, $G_{\log}$, the earlier curvature remainders, and the base term all belong to $C^{1+\varepsilon-s}$. Finally, apply \eqref{eq:5.1}, with $\beta=1+\varepsilon$, to
\[
  u-C_0\chi d_+^s
  -\sum_{k=1}^{M-1}A_k\chi d_+^{s+k\rho}
  -A_*\chi d_+^{s+1}\log d_+
\]
to obtain \eqref{eq:5.10}. Formula \eqref{eq:4.1} shows that $\alpha=1$ is the only zero in $(s,1]$, and therefore no other logarithmic terms occur.

Each $B_j$ is determined only by $C_0$ and the preceding constants $A_\ell$, so the coefficients given by \eqref{eq:5.21} and \eqref{eq:5.25} are the same along the entire boundary. All the subtracted power and logarithmic terms vanish on the boundary, and hence the final function satisfies $H|_{\partial\Omega}=C_0$.
\end{proof}

\section{Tangential Cancellation at an Orthogonal Corner}
\label{sec:tangential-cancellation}

This section concerns only $N\geq 2$; when $N=1$, the moving-plane method reduces to moving points, and no orthogonal corner arises. Let $Q\in\partial\Omega$, let $\nu$ be the inward unit normal at $Q$, and let $\tau$ be a unit tangential vector. Set
\begin{equation}
  x_{\pm}(t)=Q+t(\nu\pm\tau),\qquad t>0.
  \tag{6.1}\label{eq:orthogonal-paths}
\end{equation}

\begin{lemma}[Second-order cancellation of the distance]
\label{lem:distance-second-order-cancellation}
As $t\downarrow0$,
\begin{equation}
  d(x_{+}(t))-d(x_{-}(t))=o(t^{2}),
  \qquad
  d(x_{\pm}(t))=t+O(t^{2}).
  \tag{6.2}\label{eq:distance-second-order-cancellation}
\end{equation}
In particular, $x_{\pm}(t)\in\Omega$ for all sufficiently small $t>0$.
\end{lemma}

\begin{proof}
The signed distance function $d$ belongs to $C^{2}$ in a tubular neighborhood of the boundary and satisfies
\[
  d(Q)=0,
  \qquad
  \nabla d(Q)=\nu,
  \qquad
  |\nabla d|^{2}=1.
\]
Differentiating the last identity in the tangential direction $\tau$ gives
\[
  D^{2}d(Q)[\nu,\tau]=0.
\]
A second-order Taylor expansion at $Q$ yields
\[
  d(x_{\pm}(t))
  =t+\frac{t^{2}}{2}D^{2}d(Q)[\nu\pm\tau,\nu\pm\tau]+o(t^{2}).
\]
The difference of the quadratic Taylor terms in the two expansions is
\[
  2D^{2}d(Q)[\nu,\tau]t^{2}=0.
\]
This proves \eqref{eq:distance-second-order-cancellation}. Moreover,
$d(x_{\pm}(t))=t+O(t^{2})>0$, so both points do lie in $\Omega$.
\end{proof}

\begin{proposition}[Tangential cancellation]
\label{prop:tangential-cancellation}
Let $\psi=u/d^{s}$. Along the two paths in \eqref{eq:orthogonal-paths},
\begin{equation}
  \psi(x_{+}(t))-\psi(x_{-}(t))=o(t).
  \tag{6.3}\label{eq:psi-tangential-cancellation}
\end{equation}
If a moving plane through $Q$ has normal $\tau$, so that its reflection interchanges $x_{+}(t)$ and $x_{-}(t)$, and if
$w(x)=u(Q_{\mathrm{ref}}x)-u(x)$ is defined on the side containing $x_{-}(t)$, then
\begin{equation}
  w(x_{-}(t))=o(t^{1+s}).
  \tag{6.4}\label{eq:w-corner-upper-bound}
\end{equation}
\end{proposition}

\begin{proof}
If $\rho>1$, Theorem~5.3 directly gives $\psi\in C^{1,\varepsilon}$. Since $\psi=C_{0}$ on the boundary, its tangential derivative at $Q$ vanishes, and hence \eqref{eq:psi-tangential-cancellation} holds.

Now suppose that $\rho\leq1$. For every $\gamma=k\rho>0$, Lemma~6.1 and the mean value theorem give
\begin{align}
  \bigl|d(x_{+})^{\gamma}-d(x_{-})^{\gamma}\bigr|
  &\leq Ct^{\gamma-1}\bigl|d(x_{+})-d(x_{-})\bigr| \notag\\
  &=o(t^{1+\gamma})=o(t).
  \tag{6.5}\label{eq:power-tangential-cancellation}
\end{align}
Likewise,
\begin{align}
  &\bigl|d(x_{+})\log d(x_{+})-d(x_{-})\log d(x_{-})\bigr| \notag\\
  &\qquad\leq C(1+|\log t|)\bigl|d(x_{+})-d(x_{-})\bigr|
  =o(t^{2}|\log t|)=o(t).
  \tag{6.6}\label{eq:log-tangential-cancellation}
\end{align}
All such terms in the expansion have constant coefficients. The remainder satisfies $H\in C^{1,\varepsilon}$ and $H=C_{0}$ on the boundary; hence $\partial_{\tau}H(Q)=0$, and therefore
$H(x_{+})-H(x_{-})=o(t)$. Combining all terms proves \eqref{eq:psi-tangential-cancellation}.

Finally, the definition of the reflection gives
$w(x_{-}(t))=u(x_{+}(t))-u(x_{-}(t))$. Writing $u=d^{s}\psi$, we obtain
\begin{align*}
  w(x_{-})
  &=d(x_{+})^{s}\bigl(\psi(x_{+})-\psi(x_{-})\bigr)\\
  &\quad+\psi(x_{-})\bigl(d(x_{+})^{s}-d(x_{-})^{s}\bigr).
\end{align*}
By \eqref{eq:psi-tangential-cancellation}, the first term is $o(t^{1+s})$. For the second term,
\[
  \bigl|d(x_{+})^{s}-d(x_{-})^{s}\bigr|
  \leq Ct^{s-1}\bigl|d(x_{+})-d(x_{-})\bigr|
  =o(t^{1+s}),
\]
and $\psi$ is bounded. Thus \eqref{eq:w-corner-upper-bound} follows.
\end{proof}

\section{Moving Planes and Ball Rigidity}
\label{sec:moving-planes-ball-rigidity}

\begin{proof}[Proof of Theorem~1.2]
Fix a unit direction $e$. Using the notation of Proposition~3.3, let $m_{e}=\lambda_{*}$ and abbreviate
\[
  H=H_{m_{e}},
  \qquad
  \Sigma=\Sigma_{m_{e}},
  \qquad
  Q=Q_{m_{e}},
  \qquad
  w(x)=u(Qx)-u(x).
\]
By Proposition~3.3 and continuity, $w\geq0$ throughout $H$. The nonnegativity is not restricted to the cap $\Sigma$: if $x\in H\setminus\Sigma$, then $u(x)=0$, and hence $w(x)=u(Qx)\geq0$. Moreover, $w$ is antisymmetric with respect to $\partial H$.

Assume that $w\not\equiv0$. Since $t\mapsto t^{q}$ is strictly increasing,
\begin{equation}
  (-\Delta)^{s}w
  =\Lambda\bigl((u\circ Q)^{q}-u^{q}\bigr)\geq0
  \qquad\text{in }\Sigma.
  \tag{7.1}\label{eq:w-supersolution-critical}
\end{equation}
Thus $w$ is an entire antisymmetric supersolution with potential $c=0$; the strong maximum principle for antisymmetric functions gives
\begin{equation}
  w>0\qquad\text{in }\Sigma.
  \tag{7.2}\label{eq:w-positive-critical}
\end{equation}

The standard moving-plane geometric dichotomy for a $C^{2}$ boundary shows that, if the critical plane is not yet a symmetry plane, the first geometric critical position has one of the following two types of contact.

In the first case, the reflected boundary is internally tangent to $\partial\Omega$ away from the moving plane. Let $P\in\partial\Omega\cap H$ be a boundary point on the cap side, let $\widehat P=QP\in\partial\Omega$ be the reflected point of internal tangency, let $\eta$ be the unit normal at $P$ pointing into $\Sigma$, and set
\[
  x_{t}=P+t\eta,
  \qquad t>0.
\]
The two boundaries are tangent after reflection, so the first-order expansions of the distance functions give
\begin{equation}
  \delta(x_{t})=t+O(t^{2}),
  \qquad
  \delta(Qx_{t})=t+O(t^{2}).
  \tag{7.3}\label{eq:tangent-distance-expansions}
\end{equation}
Let $\psi=u/\delta^{s}$. By continuity of $\psi$ and because its trace on each of the two boundary portions equals $C_{0}$,
\begin{equation}
  \frac{w(x_{t})}{t^{s}}
  =\left(\frac{\delta(Qx_{t})}{t}\right)^{s}\psi(Qx_{t})
   -\left(\frac{\delta(x_{t})}{t}\right)^{s}\psi(x_{t})
  \longrightarrow0.
  \tag{7.4}\label{eq:tangent-upper-limit}
\end{equation}

On the other hand, $\Sigma$ has an interior tangent ball $B$ at $P$. By \eqref{eq:w-positive-critical} and continuity, one can choose a compact set $K\subset\Sigma$ having positive distance from both $B$ and the moving plane such that
\[
  |K|>0,
  \qquad
  \inf_{K}w>0.
\]
Here the potential is $c=0$, so $\|c\|_{\infty}=0<\lambda_{1}(B)$. Applying the antisymmetric Hopf lemma of Fall--Jarohs~\cite[Proposition~3.3]{FallJarohs2015} to \eqref{eq:w-supersolution-critical} yields a constant $c_{H}>0$ such that
\[
  w(x_{t})\geq c_{H}t^{s}
\]
for all sufficiently small $t$, contradicting \eqref{eq:tangent-upper-limit}.

In the second case, the moving plane is orthogonal to $\partial\Omega$ at a point $Q_{0}$. This case can occur only when $N\geq2$. After translating and rotating the coordinates, we may assume that
\[
  Q_{0}=0,
  \qquad
  \partial H=\{x_{1}=0\},
  \qquad
  H=\{x_{1}<0\},
\]
and let $\tau=e_{1}$ be the normal to the moving plane and $\nu=e_{2}$ the inward normal to $\partial\Omega$ at the origin. By the interior-ball property, one can choose $R>0$ sufficiently small so that
\[
  D=B_{R}(R\nu)\subset\Omega,
  \qquad
  \overline D\cap\partial\Omega=\{0\}.
\]
The center $R\nu$ lies on $\partial H$, so $D$ is symmetric with respect to the moving plane. Set $D^{*}=D\cap H$. Then $D^{*}\subset\Sigma$, and
\[
  w\geq0\quad\text{in }H,
  \qquad
  w>0\quad\text{in }D^{*},
  \qquad
  (-\Delta)^{s}w\geq0=0\cdot w\quad\text{in }D^{*}.
\]
The corner lemma of Fall--Jarohs~\cite[Lemma~4.4]{FallJarohs2015} therefore applies with the bounded potential $c=0$. Moreover, for all sufficiently small $t>0$,
\[
  |t(\nu-\tau)-R\nu|^{2}=t^{2}+(R-t)^{2}<R^{2},
\]
and hence $t(\nu-\tau)\in D^{*}$. The corner lemma yields constants $c_{1},t_{0}>0$ such that
\begin{equation}
  w\bigl(t(\nu-\tau)\bigr)\geq c_{1}t^{1+s},
  \qquad 0<t<t_{0}.
  \tag{7.5}\label{eq:corner-lower-bound}
\end{equation}
Proposition~6.2, however, gives
\[
  w\bigl(t(\nu-\tau)\bigr)=o(t^{1+s}),
\]
which contradicts \eqref{eq:corner-lower-bound}.

When $N=1$, the second, orthogonal-corner case does not arise; the first nonsymmetric critical position can only involve tangency at an endpoint, and the preceding Hopf argument has already ruled out every possibility. In conclusion, $w\not\equiv0$ leads to a contradiction, so $w\equiv0$ in $H$; antisymmetry then gives $w\equiv0$ in $\mathbb R^{N}$. Consequently, $u$ and $\Omega=\{u>0\}$ are symmetric with respect to
\[
  \Pi_{e}=\{x:x\cdot e=m_{e}\}.
\]
Since the direction $e$ is arbitrary, a symmetry plane is obtained in every direction.

It remains to rule out annuli and radial regions with holes. The moving process further gives, for every $\lambda\in(m_{e},\lambda^{+})$,
\begin{equation}
  u(Q_{\lambda}x)>u(x)
  \qquad(x\in\Sigma_{\lambda}).
  \tag{7.6}\label{eq:strict-moving-plane-monotonicity}
\end{equation}
Indeed, Proposition~3.3 first gives nonnegativity. If strict inequality failed before the critical value, the strong maximum principle would yield global symmetry with respect to $\partial H_{\lambda}$. The nonzero compactly supported function $u$ would then be symmetric with respect to two distinct parallel planes; the composition of the two reflections is a nonzero translation, contradicting the compact support of $u$.

Fix a direction $e$ and a line parallel to $e$,
\[
  L=\{z+te:t\in\mathbb R\},
  \qquad z\cdot e=0,
\]
and set
\[
  S_{L}=\{t:z+te\in\Omega\}.
\]
The set $S_{L}$ is symmetric about $m_{e}$. If $t_{2}\in S_{L}$ and $m_{e}<t_{1}<t_{2}$, take $\lambda=(t_{1}+t_{2})/2$. Then $z+t_{2}e\in\Sigma_{\lambda}$, and its reflection with respect to $\partial H_{\lambda}$ is exactly $z+t_{1}e$. By \eqref{eq:strict-moving-plane-monotonicity},
\[
  u(z+t_{1}e)>u(z+t_{2}e)>0,
\]
so $t_{1}\in S_{L}$. Taking instead $\lambda=(m_{e}+t_{2})/2$ yields, in the same way,
$u(z+m_{e}e)>u(z+t_{2}e)>0$, and hence $m_{e}\in S_{L}$. Together with symmetry about $m_{e}$, this shows that $S_{L}$ is an interval. Therefore, the intersection of $\Omega$ with every line is either an interval or the empty set, and hence $\Omega$ is convex.

Finally, define the centroid
\[
  x_{0}=\frac{1}{|\Omega|}\int_{\Omega}x\,dx.
\]
Every reflection preserving $\Omega$ fixes the centroid, and therefore $m_{e}=x_{0}\cdot e$. Thus $\Omega$ is invariant under reflection in every hyperplane through $x_{0}$. These reflections generate the full orthogonal group, so $\Omega$ is radially symmetric about $x_{0}$. A nonempty, bounded, open, convex radial set can only be a ball; hence $\Omega=B_{R}(x_{0})$. This proves the theorem.
\end{proof}

\begin{proof}[Proof of Main Theorem~1.1]
Let $u$ be the unique positive minimizer in (1.3). By (2.2), it satisfies
\[
  (-\Delta)^{s}u=\lambda_{s,p}(\Omega)u^{p-1}
  \quad\text{in }\Omega,
  \qquad
  u=0
  \quad\text{in }\Omega^{c}.
\]
Proposition~2.3 gives $u/\delta^{s}=C_{0}>0$. Standard boundedness and boundary regularity for fractional elliptic equations further yield
\[
  u\in L^{\infty}(\mathbb R^{N})\cap C^{s}(\mathbb R^{N}),
  \qquad
  u/\delta^{s}\in C(\overline\Omega).
\]
Thus $u$ satisfies all the function-space assumptions of Theorem~1.2. Taking
\[
  q=p-1\in(0,1),
  \qquad
  \Lambda=\lambda_{s,p}(\Omega)
\]
in that theorem shows that $\Omega$ is a ball.
\end{proof}

\end{document}